\RequirePackage{fix-cm}
\documentclass[smallextended]{svjour3}       
\smartqed  
\usepackage{graphicx}
\usepackage{amssymb}
\usepackage{amsmath}
\usepackage{amsfonts}

\usepackage[misc]{ifsym}
\begin{document}

\title{On 2D integro-differential systems. Stability and sensitivity analysis.
}


\author{Monika Bartkiewicz         \and
        Marek Majewski  \and
        Stanis{\l}aw Walczak
}


\institute{M. Bartkiewicz \and M. Majewski (\Letter)\at
              Faculty of Mathematics and Computer Science \\
University of Lodz, Banacha 22, 90-238 Lodz, Poland\\
              \email{bartkiew@math.uni.lodz.pl}           
\and
M. Majewski\\
\email{marmaj@math.uni.lodz.pl}
\\
\\
           S. Walczak \at
              State School of
Higher Vocational Education\\
Batorego 64 C, 96-100 Skierniewice, Poland\\
Faculty of Mathematics and Computer Science \\
University of Lodz, Banacha 22, 90-238 Lodz, Poland\\
\email{stawal@math.uni.lodz.pl}
}

\date{Received: date / Accepted: date}

\maketitle

\begin{abstract}
In the paper a two-dimensional integro-differential system is considered.
Using some variational methods we give sufficient conditions for the
existence and uniqueness of a solution to the considered system. Moreover,
we show that the system is stable and robust.
\end{abstract}

\section{Introduction}

We will denote by $Q$ the unit interval in $\mathbb{R}^{2}$, i.e.
\begin{equation}
Q=\left\{ \left(x,y\right)\in\mathbb{R}^{2}:x\in\left[0,1\right]\mbox{ and }y\in\left[0,1\right]\right\} .\label{eq:1.1}
\end{equation}
General continuous 2D differential system has the following form
form
\begin{align}
z_{xy}\left(x,y\right) & =  f(x,y,z(x,y),z_{x}(x,y),z_{y}(x,y)),\label{1.2}\\
z(x,0) & =  a(x),\text{ }z(0,y)=b(y),\text{ }a(0)=b(0),\label{1.3}
\end{align}
where $f:Q\times\mathbb{R}^{n}\times\mathbb{R}^{n}\times\mathbb{R}^{n}\rightarrow\mathbb{R}^{n}$,
$n\geq1$, $a,b:[0,1]\rightarrow\mathbb{R}^{n}$ are given functions.
The linear case of system (\ref{1.2}) can be written as
\begin{align}
z_{xy}(x,y)  &=  A_{0}(x,y)z(x,y)+A_{1}(x,y)z_{x}(x,y)+\label{1.4}
  A_{2}(x,y)z_{y}(x,y)\\&+w(x,y),\notag
\end{align}
where $A_{0},A_{1},A_{2}$ are some matrix functions of the dimension
$n\times n$, $w$ is a given $n-$dimensional vector function and
$z$ satisfies boundary conditions (\ref{1.3}). Continuous 2D systems
correspond to the discrete model of Fornasini-Marchesini type, which
has the following form (see \cite{Fornasini_Marchesini_1976})
\begin{align}
z(i+1,j+1) & =  A_{0}(i,j)z(i,j)+A_{1}(i,j)z(i+1,j)+A_{2}(i,j)z(i,j)\label{1.5}\\
 & +  w(i,j),\nonumber \\
z(i,0) & =  a(i),z(0,j)=b(j),a(0)=b(0),\label{1.6}
\end{align}
$i,j=0,1,2,\ldots$.

Two-dimensional discrete systems (\ref{1.5})--(\ref{1.6})
and continuous systems (\ref{1.2})--(\ref{1.3}) play an essential
role in mathematical modeling of many technical, physical, biological
and other phenomena. For example, in the paper by Fornasini (see \cite{Fornasini_1991})
2D space models of the form (\ref{1.5})--(\ref{1.6}) were applied
to the investigation of the process of pollution and self purification of the river. Application
of the 2D discrete models to image processing and transmission were
studied in the book of Bracewell (see \cite{bracewell_two_1994}).
The 2D continuous systems of the form (\ref{1.2})--(\ref{1.3}) were
adopted to investigation of the gas filtration model (see \cite{Bors_2012}).
Other applications of discrete and continuous 2D systems in the
theory of automatic control, stability, robotics and optimization can be found
in papers of Gałkowski, Lam, Xu and Lin \cite{galkowski_lmi_2003}; Paszke, Lam, Gałkowski, Xu and Lin \cite{paszke_robust_2004};
Lomadze, Rogers and Wood \cite{Lomadze_2008}; Kaczorek \cite{Kaczorek_2001};
Dey and Kar \cite{Dey_2011}; Singh \cite{singh_stability_2008};
Idczak and Walczak \cite{Idczak_Walczak_2000} and in the monograph of Kaczorek
\cite{Kaczorek_1985}.

In the paper, we investigate two-dimensional integro-differential system of the
form
\begin{align}
&z_{xy}(x,y)+f^{1}(x,y,z(x,y))\label{1.7}\\
&+\int_{0}^{x}\int _{0}^{y}\left(f^{2}(s,t,z(s,t))+A^{1}(s,t)z_{x}(s,t)+A^{2}(s,t)z_{y}(s,t)\right)dsdt=v(x,y)\notag
\end{align}
with the following boundary conditions
\begin{equation}
z(x,0)=0\text{ for }x\in[0,1]\text{ and } z(0,y)=0\text{ for }y\in[0,1], \label{1.8}
\end{equation}
where $f^{1},f^{2}:Q\times\mathbb{R}^{n}\rightarrow\mathbb{R}$,
$A^{1},A^{2}:Q\rightarrow\mathbb{R}^{n}\times\mathbb{R}^{n}$ are
given functions (for more details see section \ref{sec:3}). We shall
consider the above system in the space of absolutely continuous functions
of two variables. The definition and basic properties of absolutely
continuous functions defined on the interval $Q$ are presented in
section \ref{sec:2}.

In the paper we prove, under assumptions \textbf{(C1)}--\textbf{(C3)} (see section
\ref{sec:3}), that for any
square integrable function $v$ system (\ref{1.7})--(\ref{1.8})  possesses a unique solution $z_{v}$
which continuously depends on $v$ and the operator $v\mapsto z_{v}$
is differentiable in the Fréchet sense, i.e. the considered system is
well-posed and robust.

The proof of the main result is based on the global diffeomorphism
theorem (Theorem \ref{thm:2.1}). In the final part of the paper we
give an example and compare the method used in the paper with the methods
based on the contraction principle and the Schauder fixed point theorem.

\section{Preliminaries\label{sec:2}}

We begin with  the following theorem on a diffeomorphism
between Banach and Hilbert spaces
\begin{theorem}
\label{thm:2.1}Let $Z$ be a real Banach space, V be a real Hilbert
space, $F:Z\rightarrow V$ be an operator of $C^{1}$ class. If
\begin{enumerate}
\item [(a)] for any $v\in V$ the functional $\varphi(z)=\frac{1}{2}\|F(z)-v\|_{V}^{2}$
satisfies Palais-Smale condition ((PS)--condition),
\item [(b)] for any $v\in V$ the equation $F'(z)h=v$ possesses a unique
solution,
\end{enumerate}
then
\begin{enumerate}
\item [(1)] for any $v\in V$ there exists exactly one solution
$z_{v}\in Z$ to the system $F(z)=v$,
\item [(2)] the operator $V\ni v\rightarrow z_{v}\in Z$
is differentiable in the Fréchet sense.
\end{enumerate}
In other words, the operator $F$ is a diffeomorphism between Banach
space $Z$ and Hilbert space $V$.
\end{theorem}
We recall that functional $\phi$ satisfies {(PS)--condition}
if whenever there is a sequence $\left\{ z{}_{n}\right\} \subset Z$ with
$\left|\phi(z_{n})\right|\leq const$ and $\phi'(z_{n})\rightarrow0$
in $Z^{*}$, then in the closure of the set $\left\{ z_{n}:n\in N\right\} $,
there is some point $\bar{z}$ where $\phi'(\bar{z})=0$ (see \cite{aubin_applied_2006}).

From the bounded inverse theorem it follows that for any $z\in Z$
there exists a constant $\alpha_{z}>0$ such that $\|F'(z)h\|_{V}\geq\alpha_{z}\|h\|_{z}$.
Therefore it follows easily that the above theorem is equivalent to
\cite[Theorem 3.1]{dariusz_idczak_diffeomorphisms_2012}, with $f=F$.

Let us denote by $AC(Q,\mathbb{R}^{n})$ the space of absolutely continuous
vector functions $z=(z^{1},z^{2},\ldots,z^{n})$ defined on the interval
$Q$. The geometrical definition of the space $AC(Q,\mathbb{R})$ can be found in papers \cite{Berkson_1984} and \cite{Walczak_1987}.
In this paper we need necessary and sufficient conditions for
$z:Q\rightarrow\mathbb{R}^{n}$ to be absolutely continuous on $Q$ i.e.
$z\in AC(Q,\mathbb{R}^{n})$.
We have the following theorem (see \cite{Berkson_1984,Walczak_1987}).
\begin{theorem}\label{thm2.2}
A function $z$ belongs to the space $AC(Q,\mathbb{R}^{n})$ if and
only if there exist functions $l\in L^{1}(Q,\mathbb{R}^{n})$, $l^{1},l^{2}\in L^{1}([0,1],\mathbb{R}^{n})$
and a constant $c\in\mathbb{R}^{n}$ such that
\[
z(x,y)=\int _{0}^{x}\int _{0}^{y}l(s,t)dsdt+\int _{0}^{x}l^{1}(s)ds+\int _{0}^{y}l^{2}(t)dt+c.
\]
Moreover the function $z$ possesses partial derivatives $z_{x}$,
$z_{y}$, $z_{xy}$, for a.e. $(x,y)\in Q$ and
\[
z_{x}(x,y)=\int _{0}^{y}l(x,t)dt+l^{1}(x),
\]

\[
z_{y}(x,y)=\int _{0}^{x}l(s,y)ds+l^{2}(y),
\]
$$z_{xy}(x,y)=l(x,y).$$
\end{theorem}
Theorem \ref{thm2.2} follows directly from \cite[Theorem 4]{Berkson_1984}
and \cite[Proposition 3.5]{sremr_absolutely_2010} (see also \cite[Theorem 2]{Walczak_1987}, \cite[Theorem 1]{walczak_differentiability_1998}).

It is easy to check that if the function $z$ satisfies homogeneous
boundary conditions, i.e. $z(x,0)=0$ for $x\in[0,1]$ and $z(0,y)=0$
for $y\in[0,1]$ then $l^{1}=0$, $l^{2}=0$, $c=0$ and consequently
we can write
\begin{equation}
z(x,y)=\int _{0}^{x}\int _{0}^{y}l(s,t)dsdt=\int _{0}^{x}\int _{0}^{y}z_{xy}(s,t)dsdt.\label{2.1zdublowany}
\end{equation}

By $AC_{0}^{2}(Q,\mathbb{R}^{n})$ we shall denote the space of absolutely
continuous functions on the interval $Q$ which satisfy homogeneous
boundary conditions $z(x,0)=z(0,,y)=0$ for $x,y\in[0,1]$ and such
that $z_{xy}\in L^{2}(Q,\mathbb{R}^{n})$. The space $AC_{0}^{2}$
is a Hilbert space with the inner product given by formula
\begin{equation}
\langle z^{1},z^{2}\rangle=\int _{0}^{1}\int _{0}^{1}\langle z_{xy}^{1}(x,y),z_{xy}^{2}(x,y)\rangle dxdy.\label{2.1}
\end{equation}
In the space $AC_{0}^{2}(Q,\mathbb{R}^{n})$ we introduce two norms.
The first one is a classical norm given by the formula
\begin{equation}
\|z\|=\left(\int _{0}^{1}\int _{0}^{1}|z_{xy}(x,y)|^{2}dxdy\right)^{\frac{1}{2}}=\|z_{xy}\|_{L^{2}}\label{2.2}
\end{equation}
and the second  is defined by the integral with exponential weight
\begin{equation}
\|z\|_{AC_{0}^{2},m}=\left(\int _{0}^{1}\int _{0}^{1}e^{-m(x+y)}|z_{xy}(x,y)|^{2}dxdy\right)^{\frac{1}{2}}\text{, }m>0.\label{2.3}
\end{equation}
Exponential norm (\ref{2.3}) was introduced by Bielecki in
\cite{Bielecki_1956}. The space $AC_{0}^{2}(Q,\mathbb{R}^{n})$ with
norm (\ref{2.3}) will be denoted by $AC_{0,m}^{2}(Q,\mathbb{R}^{n})$.

It is easy to notice that
\[
e^{-2m}\|z\|\leq\|z\|_{AC_{0}^{2},m}\leq\|z\|.
\]
Thus the norms given by formulas (\ref{2.2}) and (\ref{2.3}) are
equivalent.

Similarly, in the space of square integrable functions on $Q$ we introduce
two equivalent norms:
\[
\|v\|=\left(\int _{0}^{1}\int _{0}^{1}|v(x,y)|^{2}dxdy\right)^{\frac{1}{2}}
\]
and
\begin{equation}
\|v\|_{L_{m}^{2}}=\left(\int _{0}^{1}\int _{0}^{1}e^{-m(x+y)}|v(x,y)|^{2}dxdy\right)^{\frac{1}{2}}.\label{2.4}
\end{equation}
The space of square integrable functions with norm (\ref{2.4}) will be
denoted by $L_{m}^{2}(Q,\mathbb{R}^{n})$.

\section{Basic assumptions and lemmas\label{sec:3}}

On the functions defining system (\ref{1.7}) we assume
that
\begin{enumerate}
\item [\textbf{(C1)}] the functions $f^{1}(\cdot,\cdot,z)$, $f^{2}(\cdot,\cdot,z)$
are measurable on $Q$ for every $z\in\mathbb{R}^{n}$ and $f^{1}(x,y,\cdot)$,
$f^{2}(x,y,\cdot)$ are continuously differentiable on $\mathbb{R}^{n}$
for a.e. $(x,y)\in Q$, the function $A^{1}(\cdot,y)$ is differentiable
for a.e. $y\in[0,1]$, the function $A^{2}(x,\cdot)$ is differentiable
for a.e. $x\in[0,1]$, the functions $A^{1},A^{2},A_{x}^{1},A_{y}^{2}$
are measurable on $Q$ and essentially bounded on $Q$, the function
$v\in L^{2}(Q,\mathbb{R}^{n})$;
\item [\textbf{(C2)}] there exist a constant $B>0$ and a function $b\in L^{2}(Q,\mathbb{R}^{+})$
such that
\[
|f^{1}(x,y,z)|,|f^{2}(x,y,z)|\leq B|z|+b(x,y)
\]
and
\[
|A^{1}(x,y)|,|A^{2}(x,y)|,|A_{x}^{1}(x,y)|,|A_{y}^{2}(x,y)|\leq B
\]
 for $z\in\mathbb{R}^{n}$ and a.e. $(x,y)\in Q$;
\item [\textbf{(C3)}] the functions $f_{z}^{1}$, $f_{z}^{2}$ are bounded on
bounded sets, i.e. for any $\varrho>0$ there exists a constant $M_{\varrho}$,
such that
\[
|f_z^{1}(x,y,z)|,|f_z^{2}(x,y,z)|\leq M_{\varrho}
\]
 for $(x,y)\in Q$ and $|z|\leq\varrho$.
\end{enumerate}
In the following lemma we prove some estimates for functions from
the space $AC_{0}^{2}(Q,\mathbb{R}^{n})$.
\begin{lemma}
\label{lem:3.1}If the function $z\in AC_{0}^{2}(Q,\mathbb{R}^{n})$
then
\begin{equation}
\|z\|_{L_{m}^{2}}\leq\frac{2}{m}\|z\|_{AC_{0,m}^{2}},\label{3.1}
\end{equation}
\begin{equation}
\|w_{0}\|_{L_{m}^{2}}\leq\frac{2}{m}\|z\|_{AC_{0,m}^{2}},\label{3.2}
\end{equation}
\begin{equation}
\|w_{1}\|_{L_{m}^{2}}\leq\frac{2}{m}\|z\|_{AC_{0,m}^{2}},\label{3.3}
\end{equation}
\begin{equation}
\|w_{2}\|_{L_{m}^{2}}\leq\frac{2}{m}\|z\|_{AC_{0,m}^{2}},\label{3.4}
\end{equation}
where $w_{0}(x,y)=\int _{0}^{x}\int _{0}^{y}|z(s,t)|dsdt$,
$w_{1}(x,y)=\int _{0}^{x}\int _{0}^{y}|z_{x}(s,t)|dsdt$,
$w_{2}(x,y)=\int _{0}^{x}\int _{0}^{y}|z_{y}(s,t)|dsdt$.\end{lemma}
\begin{remark}
The norms $\|\cdot\|_{L_{m}^{2}}$ and $\|\cdot\|_{AC_{0,m}^{2}}$
are defined by (\ref{2.3}) and (\ref{2.4}) respectively.\end{remark}
\begin{proof}
Let $z$ be an arbitrary function from the space $AC_{0}^{2}(Q,\mathbb{R}^{n})$.
By (\ref{2.1}), (\ref{2.3}) and (\ref{2.4}) we get
\begin{eqnarray*}
\|z\|_{L_{m}^{2}}^{2} & = & \int _{0}^{1}\int _{0}^{1}e^{-m(x+y)}\left|\int _{0}^{x}\int _{0}^{y}\;z_{xy}(s,t)dsdt\right|^{2}dxdy\\
 & \leq  &\int _{0}^{1}\int _{0}^{1}\left(e^{-m(x+y)}\int _{0}^{x}\int _{0}^{y}\;\left|z_{xy}(s,t)\right|^{2}dsdt\right)dxdy
\end{eqnarray*}
Integrating by parts we obtain successively
\begin{align}
&  \int_{0}^{1}\left(  \int_{0}^{1}\left(  e^{-m(x+y)}\int_{0}^{x}\int_{0}%
^{y}\left\vert z_{xy}(s,t)\right\vert ^{2}dsdt\right)  dx\right)
dy\label{3.5}\\
&  =\int_{0}^{1}\left(  \left[  \frac{-1}{m}e^{-m(x+y)}\int_{0}^{x}\int
_{0}^{y}\left\vert z_{xy}(s,t)\right\vert ^{2}dsdt\right]  _{x=0}%
^{x=1}\right.  \nonumber\\
&  -\left.  \int_{0}^{1}\left(  \frac{-1}{m}e^{-m(x+y)}\int_{0}^{y}\left\vert
z_{xy}(x,t)\right\vert ^{2}dt\right)  dx\right)  dy\nonumber\\
&  =\int_{0}^{1}\left(  \frac{-1}{m}e^{-m(1+y)}\int_{0}^{1}\int_{0}%
^{y}\left\vert z_{xy}(s,t)\right\vert ^{2}dsdt\right.  \nonumber\\
&  +\left.  \int_{0}^{1}\left(  \frac{1}{m}e^{-m(x+y)}\int_{0}^{y}\left\vert
z_{xy}(x,t)\right\vert ^{2}dt\right)  dx\right)  dy\nonumber\\
&  =\int_{0}^{1}\left(  \frac{-1}{m}e^{-m(1+y)}\int_{0}^{y}\int_{0}%
^{1}\left\vert z_{xy}(s,t)\right\vert ^{2}dsdt\right)  dy\nonumber\\
&  +\int_{0}^{1}\left(  \int_{0}^{1}\left(  \frac{1}{m}e^{-m(x+y)}\int_{0}%
^{y}\left\vert z_{xy}(x,t)\right\vert ^{2}dt\right)  dy\right)  dx\nonumber\\
&  =\left[  \frac{1}{m^{2}}e^{-m(1+y)}\int_{0}^{y}\int_{0}^{1}\left\vert
z_{xy}(s,t)\right\vert ^{2}dsdt\right]  _{y=0}^{y=1}\nonumber\\
&  +\int_{0}^{1}\left(  \frac{1}{m^{2}}e^{-m(1+y)}\int_{0}^{1}\left\vert
z_{xy}(s,y)\right\vert ^{2}ds\right)  dy\nonumber\\
&  +\int_{0}^{1}\left(  \left[  \frac{1}{m^{2}}e^{-m(x+y)}\int_{0}%
^{y}\left\vert z_{xy}(x,t)\right\vert ^{2}dt\right]  _{y=0}^{y=1}\right.
\nonumber\\
&  \left.  +\int_{0}^{1}\frac{1}{m^{2}}e^{-m(x+y)}\left\vert z_{xy}%
(x,y)\right\vert ^{2}dy\right)  dx\nonumber\\
&  \leq\frac{4}{m^{2}}\int_{0}^{1}\int_{0}^{1}e^{-m(s+t)}\left\vert
z_{xy}(s,t)\right\vert ^{2}dsdt\nonumber\\
&  =\frac{4}{m^{2}}\Vert z_{xy}\Vert_{L_{m}^{2}}^{2}=\frac{4}{m^{2}}\Vert
z\Vert_{AC_{0,m}^{2}}^{2}.\nonumber
\end{align}
Thus we proved inequality (\ref{3.1}). By the above and applying the Cauchy–Schwarz inequality we get
\begin{align*}
\Vert w_{0}\Vert_{L_{m}^{2}}^{2} &  =\int_{0}^{1}\int_{0}^{1}e^{-m(x+y)}%
\left(  \int_{0}^{x}\int_{0}^{y}|z(s,t)|dsdt\right)  ^{2}dxdy\\
&  \leq\int_{0}^{1}\int_{0}^{1}e^{-m(x+y)}\left(  \int_{0}^{x}\int_{0}%
^{y}\left(  \int_{0}^{s}\int_{0}^{t}|z_{xy}(\sigma,\tau)|d\sigma d\tau\right)
dsdt\right)  ^{2}dxdy\\
&  \leq\int_{0}^{1}\int_{0}^{1}e^{-m(x+y)}\left(  \int_{0}^{x}\int_{0}%
^{y}|z_{xy}(s,t)|dsdt\right)  ^{2}dxdy\\
&  \leq\int_{0}^{1}\int_{0}^{1}\left( e^{-m(x+y)} \int_{0}^{x}\int_{0}%
^{y}|z_{xy}(s,t)|^{2}dsdt\right)  dxdy\leq\frac{4}{m^{2}}\Vert z\Vert
_{AC_{0,m}^{2}}^{2}.
\end{align*}
Let us prove the next estimation. By (\ref{2.1zdublowany}) we have
\begin{align*}
\Vert w_{1}\Vert_{L_{m}^{2}}^{2} &  =\int_{0}^{1}\int_{0}^{1}e^{-m(x+y)}%
\left(  \int_{0}^{x}\int_{0}^{y}\;|z_{x}(s,t)|dsdt\right)  ^{2}dxdy\\
&  =\int_{0}^{1}\int_{0}^{1}e^{-m(x+y)}\left(  \int_{0}^{x}\int_{0}%
^{y}\left\vert \int_{0}^{t}\;z_{xy}(s,\tau)d\tau\right\vert dsdt\right)
^{2}dxdy\\
&  \leq\int_{0}^{1}\int_{0}^{1}e^{-m(x+y)}\left(  \int_{0}^{x}\int_{0}%
^{y}\left(  \int_{0}^{t}|z_{xy}(s,\tau)|d\tau\right)  dsdt\right)  ^{2}dxdy\\
&  \leq\int_{0}^{1}\int_{0}^{1}e^{-m(x+y)}\left(  \int_{0}^{x}\int_{0}%
^{y}\;|z_{xy}(s,t)|dsdt\right)  ^{2}dxdy.
\end{align*}
Integrating by parts  as in (\ref{3.5}), we get
\[
\|w_{1}\|_{L_{m}^{2}}^{2}\leq\frac{4}{m^{2}}\|z_{xy}\|^{2}=\frac{4}{m^{2}}\|z\|_{AC_{0,m}^{2}}^{2}.
\]
The proof of (\ref{3.4}) is similar.\qed
\end{proof}
Denote by $F:AC_{0}^{2}(Q,\mathbb{R}^{n})\rightarrow L^{2}(Q,\mathbb{R}^{n})$
the operator:
\begin{align}
F(z)(x,y) &  =z_{xy}(x,y)+f^{1}(x,y,z(x,y))\label{3.6}\\
&  +\int_{0}^{x}\int_{0}^{y}\left(  f^{2}(s,t,z(s,t))+A^{1}(s,t)z_{x}%
(s,t)+A^{2}(s,t)z_{y}(s,t)\right)  dsdt.\nonumber
\end{align}
We will prove that the norm of  $F$ is coercive.
\begin{lemma}
\label{lem:3.2}If the functions $f^{1}$, $f^{2}$, $A^{1}$, $A^{2}$
satisfy assumptions \textbf{(C1)} and  \textbf{(C2)} then the functional $z\mapsto\|F(z)\|_{L^{2}}$ is  coercive,
i.e.
\begin{equation}
\|F(z)\|_{L^{2}}\rightarrow\infty\text{ whenever }\|z\|_{AC_{0}^{2}}\rightarrow\infty.\label{3.7}
\end{equation}
\end{lemma}
\begin{proof}
Let us take $m>8B$ (cf. \textbf{(C2)}).
By (\ref{3.6}) and assumptions \textbf{(C1)}--\textbf{(C2)} we have
\[
\|F(z)\|_{L^{2}_m}\geq\|z_{xy}\|_{L_{m}^{2}}-\left(B\|z\|_{L_{m}^{2}}+B\|w_{0}\|_{L_{m}^{2}}+B\|w_{1}\|_{L_{m}^{2}}+B\|w_{2}\|_{L_{m}^{2}}\right)-D,
\]
where $D=2\|b\|_{L_{m}^{2}}$. By Lemma \ref{lem:3.1} and thanks to (\ref{2.3})
it follows that
\[
\|F(z)\|_{L^{2}_m}\geq\|z\|_{AC_{0,m}^{2}}-\frac{8B}{m}\|z\|_{AC_{0,m}^{2}}-D=\|z\|_{AC_{0,m}^{2}}\left(1-\frac{8B}{m}\right)-D.
\]
Inequality $m>8B$ implies that $\|F(z)\|_{L_{m}^{2}}\rightarrow\infty$
if $\|z\|_{AC_{0,m}^{2}}\rightarrow\infty$. Since the pairs of the
norms $\|\cdot\|_{L^{2}}$, $\|\cdot\|_{L_{m}^{2}}$ and $\|\cdot\|_{AC_{0}^{2}}$,
$\|\cdot\|_{AC_{0,m}^{2}}$ are equivalent, we conclude that (\ref{3.7})
holds.\qed
\end{proof}
Let $\{z^{k}\}_{k=0}^{\infty}\subset AC_{0}^{2}$ be an arbitrary
sequence. Denote by $\{g^{k}\}\subset L^{2}(Q,\mathbb{R}^{n})$ a
sequence of functions defined by
\begin{align}
g^{k}&(x,y)   =f^{1}(x,y,z^{k}(x,y))\label{3.8}\\
& +\int_{0}^{x}\int_{0}^{y}\left(  f^{2}(s,t,z^{k}(s,t))+A^{1}(s,t)z_{x}%
^{k}(s,t)+A^{2}(s,t)z_{y}^{k}(s,t)\right)  dsdt-v(x,y),\nonumber
\end{align}
for $k=0,1,2,\ldots$.
\begin{lemma}\label{lem:3.4}
If
\begin{enumerate}
\item the functions $f^{1}$, $f^{2}$, $A^{1}$, $A^{2}$ satisfy assumptions
\textbf{(C1)} and \textbf{(C2)};
\item the sequence $\{z^{k}\}_{k=0}^{\infty}\subset AC_{0}^{2}(Q,\mathbb{R}^{n})$
tends to $z^{0}\in AC_{0}^{2}(Q,\mathbb{R}^{n})$ weakly in $AC_{0}^{2}(Q,\mathbb{R}^{n})$
\end{enumerate}
then
\begin{enumerate}
\item [(a)] the sequence of functions $\{z^{k}\}$ tends uniformly
to $z^{0}$ on the interval $Q$;
\item [(b)] the sequence $\{g^{k}\}$ tends to $g^{0}$
for $(x,y)\in Q$ a.e.
\end{enumerate}
Moreover, there exists a function $b^{0}\in L^{2}(Q,\mathbb{R}^{+})$
such that
\[
|g^{k}(x,y)|\leq b^{0}(x,y)
\]
for a.e. $(x,y)\in Q$ and $k=1,2,\ldots$.

\end{lemma}
\begin{proof}
We first prove that the weak convergence of the sequence $\{z^{k}\}$
to $z^{0}$ in the space $AC_{0}^{2}(Q,\mathbb{R}^{n})$ implies the
uniform convergence of the sequence  $\{z^{k}\}$
to $z^{0}$ on the interval $Q$. By the definition of the inner
product (see (\ref{2.1})) the weak convergence of the sequence $\{z^{k}\}$
to $z^{0}$ in the space $AC_{0}^{2}(Q,\mathbb{R}^{n})$ is equivalent
to the weak convergence of mixed second order derivatives $\{z_{xy}^{k}\}$
to $z_{xy}^{0}$ in the space $L^{2}(Q,\mathbb{R}^{n})$. Without loss
of generality we can assume that $z^{0}=0$. Suppose that $z^{k}$
does not converge uniformly to $z^{0}=0$ while it converges
to $0$ weakly in $AC_{0}^{2}(Q,\mathbb{R}^{n})$. Therefore there
exists $\varepsilon_{0}>0$ such that for any $n\in\mathbb{N}$ there
is a point $(x^{n},y^{n})\in Q$ such that
\begin{equation}
|z^{n}(x^{n},y^{n})|>\varepsilon_{0}.\label{3.9}
\end{equation}
The sequence $\left\{\left(x^{n},y^{n}\right)\right\}\subset Q$ is compact. Passing if necessary
to a subsequence we can assume, that $(x^{n},y^{n})$ tends to some
$(\tilde{x},\tilde{y})\in Q$. Denote by $\chi^{n}$ the characteristic
function of the interval
\[
\{(x,y)\in Q:0\leq x<x^{n},0\leq y<y^{n}\}
\]
 and by $\tilde{\chi}$ the characteristic function of the interval
\[
\{(x,y)\in Q:0\leq x<\tilde{x},0\leq y<\tilde{y}\}.
\]
It is easy to notice that $\chi^{n}$ tends to $\tilde{\chi}$ on
$Q$ a.e. This implies the following inequalities
\begin{align*}
&\lim_{n\rightarrow\infty}|z^{n}(x^{n},y^{n})|  \leq\lim_{n\rightarrow\infty
}|z^{n}(x^{n},y^{n})-z^{n}(\tilde{x},\tilde{y})|+\lim_{n\rightarrow\infty
}|z^{n}(\tilde{x},\tilde{y})|\\
=&\lim_{n\rightarrow\infty}\left\vert \int_{0}^{1}\int_{0}^{1}\chi
^{n}(s,t)z_{xy}^{n}(s,t)dsdt-\int_{0}^{1}\int_{0}^{1}\tilde{\chi}%
(s,t)z_{xy}^{n}(s,t)dsdt\right\vert \\
+&\lim_{n\rightarrow\infty}\left\vert \int_{0}^{1}\int_{0}^{1}\tilde{\chi
}(s,t)z_{xy}^{n}(s,t)dsdt\right\vert .
\end{align*}
Since $z_{xy}^{n}$ tends to zero weakly in $L^{2}(Q,\mathbb{R}^{n})$
the last limit is equal zero. Therefore
\begin{align*}
&\lim_{n\rightarrow\infty}|z^{n}(x^{n},y^{n})| \leq\lim_{n\rightarrow\infty
}\int_{0}^{1}\int_{0}^{1}|\chi^{n}(s,t)-\tilde{\chi}(s,t)||z_{xy}%
^{n}(s,t)|dsdt\\
\leq&\lim_{n\rightarrow\infty}\left(  \int_{0}^{1}\int_{0}^{1}|\chi
^{n}(s,t)-\tilde{\chi}(s,t)|^{2}dsdt\right)  ^{\frac{1}{2}}\cdot\left(
\int_{0}^{1}\int_{0}^{1}|z_{xy}^{n}(s,t)|^{2}dsdt\right)  ^{\frac{1}{2}}\\
\leq C\cdot&\lim_{n\rightarrow\infty}\left(  \int_{0}^{1}\int_{0}^{1}%
|\chi^{n}(s,t)-\tilde{\chi}(s,t)|^{2}dsdt\right)  ^{\frac{1}{2}}=0,
\end{align*}
where $C>0$ is some constant such that $\|z_{xy}^{n}\|\leq C$. Consequently,\break
$\lim _{n\rightarrow\infty}|z^{n}(x^{n},y^{n})|=0$. This contradicts
our assumption (\ref{3.9}). Thus $z^{k}$ tends to $z^{0}$ uniformly
on $Q$. 

Next we prove  the assertion (b) of Lemma \ref{lem:3.4}. By assumptions \textbf{(C1)} and \textbf{(C2)} we infer that
\begin{equation}
\lim _{k\rightarrow\infty}f^{1}(x,y,z^{k}(x,y))=f^{1}(x,y,z^{0}(x,y))\label{3.10}
\end{equation}
for a.e. $(x,y)\in Q$,
\[
\lim _{k\rightarrow\infty}\int _{0}^{x}\int _{0}^{y}f^{2}(s,t,z^{k}(s,t))dsdt=\int _{0}^{x}\int _{0}^{y}f^{2}(s,t,z^{0}(s,t))dsdt
\]
for a.e. $(x,y)\in Q$ and that there exists a function $b^{1}\in L^{2}(Q,\mathbb{R}^{+})$
such that
\[
\left|f^{1}(x,y,z^{k}(x,y))\right|,\left|\int _{0}^{x}\int _{0}^{y}f^{2}(s,t,z^{k}(s,t))dsdt\right|\leq b^{1}(x,y)
\]
 for a.e. $(x,y)\in Q$ and $k=1,2,\ldots$. Integrating by parts
and taking into account assumption \textbf{(C1)} and Fubini's theorem we obtain
\begin{align}
\int_{0}^{x}\int_{0}^{y}A^{1}(s,t)z_{x}^{k}(s,t)dsdt  & =\int_{0}^{y}\left(
\int_{0}^{x}A^{1}(s,t)z_{x}^{k}(s,t)ds\right)  dt\label{3.11}\\
& =\int_{0}^{y}A^{1}(x,t)z^{k}(x,t)dt-\int_{0}^{x}\int_{0}^{y}A_{x}%
^{1}(s,t)z^{k}(s,t)dsdt\nonumber
\end{align}
for $k=0,1,...$. Since $z^{k}$ tends to $z^{0}$ uniformly on $Q$ provided
that $z^{k}$ converges to $z^{0}$ weakly in $AC_{0}^{2}(Q,\mathbb{R}^{n})$,
we get thanks to (\ref{3.11}) that
\begin{align}
&\lim_{k\rightarrow\infty}\int_{0}^{x}\int_{0}^{y}A^{1}(s,t)z_{x}^{k}(s,t)dsdt
=\label{3.12}\\
& =\int_{0}^{y}\lim_{k\rightarrow\infty}A^{1}(x,t)z^{k}(x,t)dt-\int_{0}%
^{x}\int_{0}^{y}\lim_{k\rightarrow\infty}A_{x}^{1}(s,t)z^{k}%
(s,t)dsdt\nonumber\\
& =\int_{0}^{y}A^{1}(x,t)z^{0}(x,t)dt-\int_{0}^{x}\int_{0}^{y}A_{x}%
^{1}(s,t)z^{0}(s,t)dsdt\nonumber\\
& =\int_{0}^{x}\int_{0}^{y}A^{1}(s,t)z_{x}^{0}(s,t)dsdt\nonumber
\end{align}
for $(x,y)\in Q$. Similarly, we can show that
\begin{equation}
\lim _{k\rightarrow\infty}\int _{0}^{x}\int _{0}^{y}A^{2}(s,t)z_{y}^{k}(s,t)dsdt=\int _{0}^{x}\int _{0}^{y}A^{2}(s,t)z_{y}^{0}(s,t)dsdt\label{3.13}
\end{equation}
for $(x,y)\in Q$. By (\ref{3.11}) it is easy to notice that
\begin{equation}
\left|\int _{0}^{x}\int _{0}^{y}A^{1}(s,t)z_{x}^{k}(s,t)dsdt\right|\leq C^{1}\label{3.14}
\end{equation}
for some constant $C^{1}>0$, all $(x,y)\in Q$ and $k=1,2,\ldots$.
Similar estimation holds for the integral $\int _{0}^{x}\int _{0}^{y}A^{2}(s,t)z_{y}^{k}(s,t)dsdt$.
From (\ref{3.10}), (\ref{3.12}), (\ref{3.13}) and (\ref{3.14}) it
follows that
\[
\lim _{k\rightarrow\infty}g^{k}(x,y)=g^{0}(x,y)
\]
for a.e. $(x,y)\in Q$ Moreover, there exists a function $b^0\in L^{2}(Q,\mathbb{R}^{+})$
such that $|g^{k}(x,y)|\leq b^0(x,y)$ for $k=1,2,\ldots$ and
a.e. $(x,y)\in Q$. This completes the proof.\qed\end{proof}

\section{Main result and example}

Let us consider a functional $\varphi:AC_{0}^{2}\rightarrow\mathbb{R}$
given by the formula
\begin{equation}
\varphi\left(z\right)=\frac{1}{2}\left\Vert F(z)-v\right\Vert _{L^{2}}^{2},\label{eq:4.1}
\end{equation}
where $F$ is the operator defined by (3.6) and $v$ is a fixed function
from the space $L^{2}\left(Q,\mathbb{R}^{n}\right)$. We begin by
proving some lemmas.
\begin{lemma}
\label{lem:4.1}If the functions $f^1,f^2,A^1,A^2$ satisfy  assumptions \textbf{(C1)}--\textbf{(C2)},
then the functional $\varphi$ given by (\ref{eq:4.1}) satisfies (PS)--condition.\end{lemma}
\begin{proof}
Let $\{z^{k}\}\subset AC_{0}^{2}$ be an arbitrary (PS)--sequence for
the functional $\varphi$. By Lemma \ref{lem:3.2} $\varphi$ is coercive. It
implies that the sequence $\{z^{k}\}$ is weakly compact in $AC_{0}^{2}$.
Passing if necessary to a subsequence we can assume, that $z^{k}$
tends to some $z^{0}$ weakly in $AC_{0}^{2}$. We claim that $\{z^{k}\}$
is compact with respect to the norm topology of the space $AC_{0}^{2}$.
Thanks to assumptions \textbf{(C1)}--\textbf{(C2)} it is easy to check that the functional $\varphi$
is Fréchet differentiable and
\begin{align}
\left\langle \varphi^{\prime}\left(  z^{k}\right)  ,h\right\rangle  &
=\int_{0}^{1}\int_{0}^{1}\left\langle h_{xy}\left(  x,y\right)  +f_{z}%
^{1}\left(  x,y,z^{k}\left(  x,y\right)  \right)  h\left(  x,y\right)
\right.  \nonumber\\
& +\int_{0}^{x}\int_{0}^{y}\left(  f_{z}^{2}\left(  s,t,z^{k}\left(
s,t\right)  \right)  h\left(  s,t\right)  +A^{1}\left(  s,t\right)
h_{x}\left(  s,t\right)  \right.  \label{eq:4.2}\\
& +\left.  \left.  A^{2}\left(  s,t\right)  h_{y}\left(  s,t\right)  \right)
dsdt,z_{xy}^{k}\left(  x,y\right)  +g^{k}\left(  x,y\right)  \right\rangle
dxdy,\nonumber
\end{align}
where the sequence $\{g^{k}\}\subset L^{2}\left(Q,\mathbb{R}^{n}\right)$
is given by formula (\ref{3.8}). Let us put $h^k-z^k-z^0$, $k=1,2,...$. From (\ref{eq:4.2}) it follows that
\begin{align}
\left\langle \varphi^{\prime}\left(  z^{k}\right)  -\varphi^{\prime}\left(
z^{0}\right)  ,z^{k}-z^{0}\right\rangle  & =\left\langle z_{xy}^{k}-z_{xy}%
^{0},h_{xy}^{k}\right\rangle +\sum_{i=1}^{5}V^{i}(z^{k})\nonumber\\
& =\left\Vert z^{k}-z^{0}\right\Vert _{AC_{0}^{2}}^{2}+\sum_{i=1}^{5}%
V^{i}\left(  z^{k}\right)  ,\label{eq:4.3}%
\end{align}

where
\begin{align*}
V^{1}\left(  z^{k}\right)    & =\left\langle z_{xy}^{k}-z_{xy}^{0},g^{k}%
-g^{0}\right\rangle \\
& =\int_{0}^{1}\int_{0}^{1}\left\langle z_{xy}^{k}\left(  x,y\right)
-z_{xy}^{0}\left(  x,y\right)  ,g^{k}\left(  x,y\right)  -g^{0}\left(
x,y\right)  \right\rangle dxdy,
\end{align*}%
\begin{align*}
V^{2}\left(  z^{k}\right)&=\int_{0}^{1}\int_{0}^{1}\left\langle  f_{z}%
^{1}\left(  x,y,z^{k}\left(  x,y\right)  \right)  \left(  z^{k}\left(
x,y\right)  -z^{0}\left(  x,y\right)  \right)  \right.  \\
& +\int_{0}^{x}\int_{0}^{y}f_{z}^{2}\left(  s,t,z^{k}\left(  s,t\right)
\right)  \left(  z^{k}\left(  s,t\right)  -z^{0}\left(  s,t\right)  \right)
dsdt,\\
& \left.  z_{xy}^{k}\left(  x,y\right)  +g^{k}\left(  x,y\right)
\right\rangle dxdy,
\end{align*}%
\begin{align*}
V^{3}\left(  z^{k}\right)    & =-\int_{0}^{1}\int_{0}^{1}\left\langle
f_{z}^{1}\left(  x,y,z^{0}\left(  x,y\right)  \right)  \left(  z^{k}\left(
x,y\right)  -z^{0}\left(  x,y\right)  \right)  \right.  \\
& +\int_{0}^{x}\int_{0}^{y}f_{z}^{2}\left(  s,t,z^{0}\left(  s,t\right)
\right)  \left(  z^{k}\left(  s,t\right)  -z^{0}\left(  s,t\right)  \right)
dsdt,\\
& \left.  z_{xy}^{0}\left(  x,y\right)  -g^{0}\left(  x,y\right)
\right\rangle dxdy,
\end{align*}%
\begin{align*}
V^{4}\left(  z^{k}\right)    & =\int_{0}^{1}\int_{0}^{1}\left\langle \int
_{0}^{x}\int_{0}^{y}A^{1}\left(  s,t\right)  \left(  z_{x}^{k}\left(
s,t\right)  -z_{x}^{0}\left(  s,t\right)  \right)  dsdt,\right.  \\
& \left.  z_{xy}^{k}\left(  x,y\right)  +g^{k}\left(  x,y\right)
\right\rangle dxdy\\
& +\int_{0}^{1}\int_{0}^{1}\left\langle \int_{0}^{x}\int_{0}^{y}A^{2}\left(
s,t\right)  \left(  z_{y}^{k}\left(  s,t\right)  -z_{y}^{0}\left(  s,t\right)
\right)  dsdt,\right.  \\
& \left.  z_{xy}^{k}\left(  x,y\right)  +g^{k}\left(  x,y\right)
\right\rangle dxdy,
\end{align*}%
\begin{align*}
V^{5}\left(  z^{k}\right)    & =-\int_{0}^{1}\int_{0}^{1}\left\langle \int
_{0}^{x}\int_{0}^{y}A^{1}\left(  s,t\right)  \left(  z_{x}^{k}\left(
s,t\right)  -z_{x}^{0}\left(  s,t\right)  \right)  dsdt,\right.  \\
& \left.  z_{xy}^{0}\left(  x,y\right)  +g^{0}\left(  x,y\right)
\right\rangle dxdy\\
& -\int_{0}^{1}\int_{0}^{1}\left\langle \int_{0}^{x}\int_{0}^{y}A^{2}\left(
s,t\right)  \left(  z_{y}^{k}\left(  s,t\right)  -z_{y}^{0}\left(  s,t\right)
\right)  dsdt,\right.  \\
& \left.  z_{xy}^{0}\left(  x,y\right)  +g^{0}\left(  x,y\right)
\right\rangle dxdy.
\end{align*}

By the Cauchy-Schwarz inequality we have the following estimation
\begin{align*}
\left\vert V^{1}\left(  z^{k}\right)  \right\vert ^{2}  & \leq\int_{0}^{1}%
\int_{0}^{1}\left\vert z_{xy}^{k}\left(  x,y\right)  -z_{xy}^{0}\left(
x,y\right)  \right\vert ^{2}dxdy\\
& \cdot\int_{0}^{1}\int_{0}^{1}\left\vert g^{k}\left(  x,y\right)
-g^{0}\left(  x,y\right)  \right\vert ^{2}dxdy.
\end{align*}
Since $z_{xy}^{k}-z_{xy}^{0}$ converges weakly to zero in $L^{2}\left(Q,\mathbb{R}^{n}\right)$,
therefore there exists a constant $C>0$ such that
\[
\left|V^{1}\left(z^{k}\right)\right|^{2}\leq C\int_{0}^{1}\int_{0}^{1}\left|g^{k}\left(x,y\right)-g^{0}\left(x,y\right)\right|^{2}dxdy.
\]

By Lemma \ref{lem:3.2} and Lebesgue dominated convergence theorem it follows
that $V^{1}\left(z^{k}\right)\rightarrow0$ as $k\rightarrow\infty$.
We have proved that $z^{k}\left(x,y\right)$ tends to $z^{0}\left(x,y\right)$
uniformly on $Q$ (see Lemma \ref{lem:3.2}). Therefore, it is easy to notice
that $V^{2}\left(z^{k}\right)$ and $V^{3}\left(z^{k}\right)$ converge
to zero as $k\rightarrow\infty$.

Let us consider the functional $V^{4}$. By (\ref{3.11}) we have
\begin{align*}
V^{4}\left(  z^{k}\right)    & =\int_{0}^{1}\int_{0}^{1}\left\langle \int
_{0}^{y}A^{1}\left(  x,t\right)  \left(  z^{k}\left(  x,t\right)
-z^{0}\left(  x,t\right)  \right)  dt\right.  \\
& -\int_{0}^{x}\int_{0}^{y}A_{x}^{1}\left(  s,t\right)  \left(  z^{k}\left(
s,t\right)  -z^{0}\left(  s,t\right)  \right)  dsdt,\\
& \left.  z_{xy}^{k}\left(  x,y\right)  +g^{k}\left(  x,y\right)
\right\rangle dxdy\\
& +\int_{0}^{1}\int_{0}^{1}\left\langle \int_{0}^{x}A^{2}\left(  s,y\right)
\left(  z^{k}\left(  s,y\right)  -z^{0}\left(  s,y\right)  \right)  ds\right.
\\
& -\int_{0}^{x}\int_{0}^{y}A_{y}^{2}\left(  s,t\right)  \left(  z^{k}\left(
s,t\right)  -z^{0}\left(  s,t\right)  \right)  dsdt,\\
& \left.  z_{xy}^{k}\left(  x,y\right)  +g^{k}\left(  x,y\right)
\right\rangle dxdy.
\end{align*}
Using the Cauchy-Schwarz inequality and Lemma \ref{lem:3.2} it is easy to show that $V^{4}\left(z^{k}\right)\rightarrow0$
as $k\rightarrow\infty$.

Similar considerations can be applied to $V^{5}\left(z^{k}\right)$. Thus $\lim_{k\rightarrow\infty}\sum_{i=1}^{5}V^{i}\left(z^{k}\right)=0$.

Now, let us observe that
\[
\lim_{k\rightarrow\infty}\varphi^{\prime}\left(z^{k}\right)\left(z^{k}-z^{0}\right)=0
\]
because $\{z^{k}\}$ is the (PS)--sequence for the functional $\varphi$
and the sequence $\left\{ z^{k}-z^{0}\right\} $ is bounded. Moreover,
\[
\lim_{k\rightarrow\infty}\varphi^{\prime}\left(z^{0}\right)\left(z^{k}-z^{0}\right)=0
\]
since $z^{k}$ tends weakly to $z^{0}$ in $AC_{0}^{2}$. Combining
these equalities and (\ref{eq:4.3}) we conclude that $$\lim_{k\rightarrow\infty}\left\Vert z^{k}-z^{0}\right\Vert _{AC_{0}^{2}}^{2}=0.$$
This gives us the desired conclusion that the functional $\varphi$
given by (\ref{eq:4.1}) satisfies (PS)--condition.\qed
\end{proof}
Next, we prove the following
\begin{lemma}
\label{lem:4.2}If the functions $f^{1}$,$f^{2}$, \textup{$A^{1}$,
$A^{2}$ satisfy assumptions \textbf{(C1)}--\textbf{(C3)} then for any $v\in L^{2}\left(Q,\mathbb{R}^{n}\right)$
there exists a unique solution $h_{v}\in AC_{0}^{2}$ to the system
\begin{equation}
F'\left(z^{0}\right)h=v,\label{eq:4.4}
\end{equation}
where the operator $F:AC_{0}^{2}\rightarrow L^{2}\left(Q,\mathbb{R}^{n}\right)$
is given by (\ref{3.6}) and $z^{0}\in AC_{0}^{2}$ is an arbitrary function.}\end{lemma}
\begin{proof}
Let us put
\[
h\left(x,y\right)=\int_{0}^{x}\int_{0}^{y}g\left(s,t\right)dsdt,
\]
where $g\in L^{2}\left(Q,\mathbb{R}^{n}\right)$.
Substituting the above into (\ref{eq:4.4}) we obtain
\[
Hg=v,
\]
where
\begin{align*}
Hg\left(  x,y\right)    & =g\left(  x,y\right)  +f_{z}^{1}\left(
x,y,z^{0}\left(  x,y\right)  \right)  \cdot\int_{0}^{x}\int_{0}^{y}g\left(
s,t\right)  dsdt\\
& +\int_{0}^{x}\int_{0}^{y}\left(  f_{z}^{2}\left(  s,t,z^{0}\left(
s,t\right)  \right)  \int_{0}^{s}\int_{0}^{t}g\left(  \sigma,\tau\right)
d\sigma d\tau\right.  \\
& +\left.  A^{1}\left(  s,t\right)  \int_{0}^{t}g\left(  s,\tau\right)
d\tau+A^{2}\left(  s,t\right)  \int_{0}^{s}g\left(  \sigma,t\right)
d\sigma\right)  dsdt.
\end{align*}

Let us denote by $\tilde{H}$ the operator defined by
\begin{equation}
\tilde{H}g=Hg-g-v.\label{eq:4.5}
\end{equation}
We will restrict our investigation of the operator $\tilde{H}$ to
the space $L_{m}^{2}\left(Q,\mathbb{R}^{n}\right)$. We prove that
for sufficiently large $m>0$ the mapping $\tilde{H}$ is contracting
with respect to the norm $\left\Vert \cdot\right\Vert _{L_{m}^{2}}$
defined by (\ref{2.4}).
Under assumptions \textbf{(C2)} and \textbf{(C3)}, there exists a constant $d>0$ such
that
\begin{align*}
& \left\Vert \tilde{H}\left(  g^{1}-g^{2}\right)  \right\Vert _{L_{m}^{2}}\\
& \leq d\left(  \int_{0}^{1}\int_{0}^{1}\left(  e^{-m\left(  x+y\right)  }%
\int_{0}^{x}\int_{0}^{y}\left\vert g^{1}\left(  s,t\right)  -g^{2}\left(
s,t\right)  \right\vert ^{2}dsdt\right)  dxdy\right)  ^{\frac{1}{2}}\\
& +d\left(  \int_{0}^{1}\int_{0}^{1}\left(  e^{-m\left(  x+y\right)  }\int
_{0}^{x}\int_{0}^{y}\left(  \int_{0}^{s}\int_{0}^{t}\left\vert \left(
g^{1}\left(  \sigma,\tau\right)  -g^{2}\left(  \sigma,\tau\right)  \right)
\right\vert ^{2}d\sigma d\tau\right)  dsdt\right)  dxdy\right)  ^{\frac{1}{2}%
}\\
& +d\left(  \int_{0}^{1}\int_{0}^{1}\left(  e^{-m\left(  x+y\right)  }\int
_{0}^{x}\int_{0}^{y}\left(  \int_{0}^{t}\left\vert \left(  g^{1}\left(
s,\tau\right)  -g^{2}\left(  s,\tau\right)  \right)  \right\vert ^{2}%
d\tau\right)  dsdt\right)  dxdy\right)  ^{\frac{1}{2}}\\
& +d\left(  \int_{0}^{1}\int_{0}^{1}\left(  e^{-m\left(  x+y\right)  }\int
_{0}^{x}\int_{0}^{y}\left(  \int_{0}^{s}\left\vert \left(  g^{1}\left(
\sigma,t\right)  -g^{2}\left(  \sigma,t\right)  \right)  \right\vert
^{2}d\sigma\right)  dsdt\right)  dxdy\right)  ^{\frac{1}{2}}\\
& \leq4d\left(  \int_{0}^{1}\int_{0}^{1}\left(  e^{-m\left(  x+y\right)  }%
\int_{0}^{x}\int_{0}^{y}\left\vert g^{1}\left(  s,t\right)  -g^{2}\left(
s,t\right)  \right\vert ^{2}dsdt\right)  dxdy\right)  ^{\frac{1}{2}}.
\end{align*}
Integrating by parts twice, in much the same way as in the proof of
inequality (\ref{3.5}), we obtain
\[
\left\Vert \tilde{H}\left(g^{1}-g^{2}\right)\right\Vert _{L_{m}^{2}}\leq\frac{4d}{m^{2}}\left\Vert g^{1}-g^{2}\right\Vert _{L_{m}^{2}}.
\]
Hence for sufficiently large $m$, i.e. $m>2\sqrt{d}$, the operator
$\tilde{H}$ is contracting and, consequently, has a unique fixed
point. It means that, there exists exactly one point $g^{0}\in L^{2}\left(Q,\mathbb{R}^{n}\right)$
such that $g^{0}=\tilde{H}g^{0}$. By (\ref{eq:4.5}) we get ${H}g^{0}=v$
and it follows easily that a function $h_{v}$ given by
\[
h_{v}\left(x,y\right)=\int_{0}^{x}\int_{0}^{y}g^{0}\left(s,t\right)dsdt
\]
is a solution of (\ref{eq:4.4}) for fixed $v\in L^{2}\left(Q,\mathbb{R}^{n}\right)$.\qed
\end{proof}
We are now in a position to show the main result of the work.
\begin{theorem}
\label{thm:4.1}If the functions $f^{1}$,$f^{2}$, $A^{1}$,
$A^{2}$ satisfy assumptions \textbf{(C1)}--\textbf{(C3)} then for any $v\in L^{2}\left(Q,\mathbb{R}^{n}\right)$
the integro-differential system (\ref{1.7})--(\ref{1.8}) has a unique solution $z_{v}\in AC_{0}^{2}$.
The solution $z_{v}$ continuously depends on $v$ with respect to
the norm topology in the spaces \textup{$L^{2}\left(Q,\mathbb{R}^{n}\right)$
and $AC_{0}^{2}$. Moreover, the operator
\[
L^{2}\left(Q,\mathbb{R}^{n}\right)\ni v\mapsto z_{v}\in AC_{0}^{2}
\]
is differentiable (in Fr\'{e}chet sense).}\end{theorem}
\begin{proof}
If follows from Lemmas \ref{lem:4.1} and \ref{lem:4.2} that the
operator $F$ given by (\ref{3.6}) meets assumptions of Theorem \ref{thm:2.1}. Thus
system (\ref{1.7})--(\ref{1.8})  has a solution $z_{v}$ which satisfies the requirements
of our theorem.\qed
\end{proof}
We now give an example of integro-differential system of the form
(\ref{1.7})--(\ref{1.8}) which satisfies assumptions of Theorem \ref{thm:4.1}. For simplicity we put $n=1$.
\begin{example}
Consider 2D integro-differential system
\begin{align}
& z_{xy}\left(  x,y\right)  +w^{1}\left(  x,y\right)  \left(  \frac
{z^{3}\left(  x,y\right)  }{1+z^{2}\left(  x,y\right)  }+\psi^{1}%
(z(x,y)\right)  \nonumber\\
& +\int_{0}^{x}\int_{0}^{y}\left(  w^{2}\left(  s,t\right)  \frac{z\left(
s,t\right)  -1}{1+z^{2}\left(  x,y\right)  }+\psi^{2}(z(x,y))\right.
\nonumber\\
& \left.  +A^{1}\left(  s,t\right)  z_{x}\left(  s,t\right)  +A^{2}\left(
s,t\right)  z_{y}\left(  s,t\right)  dsdt\right)   =v\left(  x,y\right)  ,\label{eq:4.6}%
\end{align}
where $w^{1},w^{2},A^{1},A^{2}$ are some polynomials,
$v\in L^{2}\left(Q,\mathbb{R}\right)$ and $\psi^1,\psi^2$ are some $C^1-$class functions with unbounded derivatives. For example one can take $\psi^1(z)=\cos z^k$ and $\psi^2(z)=\sin z^l$, where $k,l>1$. This simple and theoretical example allows us to emphasize
the difference between our work and some other methods of nonlinear
analysis.

It is easy to see that system (\ref{eq:4.6}) satisfies assumptions
\textbf{(C1)}--\textbf{(C3)}. Hence by Theorem \ref{thm:4.1} for any $v\in L^{2}\left(Q,\mathbb{R}\right)$
there exists a solution $z_{v}\in AC_{0}^{2}$ to the system (\ref{eq:4.6})
with the following properties:
\begin{enumerate}
\item the solution $z_{v}$ is unique,
\item $z_{v}$ continuously depends on $v$ with respect to the norm topology
of the spaces $L^{2}\left(Q,\mathbb{R}^{n}\right)$ and $AC_{0}^{2}$,
i.e. system (\ref{eq:4.6}) is stable,
\item the operator $L^{2}\left(Q,\mathbb{R}\right)\ni v\mapsto z_{v}\in AC_{0}^{2}$
is differentiable in Fréchet sense, i.e. system (\ref{eq:4.6}) is
robust.
\end{enumerate}
Let us notice that the functions $f^{1}\left(x,y,z\right)=w^{1}\left(x,y\right)\left(\frac{z^{3}}{1+z^{2}}+\psi^1(z)\right)$ and
$f^{2}\left(x,y,z\right)=w^{2}\left(x,y\right)\frac{z-1}{1+z^{2}}\allowbreak+\psi^2(z)$
are not  Lipschitz functions ($\sin z^{l}$ and $\cos z^l$ with $k,l\geq 1$ have ''fast variation''
when $\left|z\right|\rightarrow\infty$) and consequently we cannot
apply the Banach contraction principle. In this case the Schauder
fixed point theory may be applicable. But even using sophisticated
fixed point theorems we get only the existence of a solution to system
(\ref{eq:4.6}) and can hardly say anything related to properties
(1)--(3).
\end{example}

\section{Concluding remarks}

In the paper two-dimensional integro-differential system was investigated.
The main result of this work is theorem \ref{thm:4.1} on the stability
and robustness of a solution to considered system (\ref{1.7})--(\ref{1.8}).
As far as we know 2D integro-differential systems have not been studied
before. One-dimensional integro-differential systems described by
ordinary differential operators were examined in many works (see monogrph
\cite{lakshmikantham_theory_1995} and references therein). It is
important to notice that integro-differential operators can be used
in mathematical modeling of systems with ``memory'', i.e. systems
where the state at each moment $t$ depends on its behavior on some
interval $[t_{0},t)$. In our opinion 2D integro-differential systems
have the potential to play a similar role.

\bibliographystyle{apalike}
\bibliography{baza}

\end{document}